\documentclass[12pt]{amsart}
\usepackage{amsfonts}
\usepackage{amssymb}
\usepackage{enumerate}
\usepackage{graphicx}

\makeatletter
\@namedef{subjclassname@2010}{  \textup{2010} Mathematics Subject Classification}
\makeatother
\newtheorem{theorem}{Theorem}[]
\newtheorem{lemma}[]{Lemma}

\newtheorem{corollary}[]{Corollary}

\theoremstyle{definition}

\newtheorem{remark}[]{Remark}

\numberwithin{equation}{section}
\begin{document}
\title[Monotone nonexpansive mappings]{The Knaster-Tarski theorem \\
versus \\
monotone nonexpansive mappings}
\author[R. Esp\'{\i}nola]{Rafael Esp\'{\i}nola}
\address{Departamento de An\'{a}lisis Matem\'{a}tico - IMUS, Universidad de
Sevilla, Apdo. de Correos 1160, 41080 Sevilla, Spain}
\email{espinola@us.es}
\author[A. Wi\'{s}nicki]{Andrzej Wi\'{s}nicki}
\address{Department of Mathematics, Rzesz\'{o}w University of Technology,
Al. Powsta\'{n}c\'{o}w Warszawy 12, 35-959 Rzesz\'{o}w, Poland}
\email{awisnicki@prz.edu.pl}
\date{}

\begin{abstract}
Let $X$ be a partially ordered set with the property that each family of
order intervals of the form $[a,b],[a,\rightarrow )$ with the finite
intersection property has a nonempty intersection. We show that every
directed subset of $X$ has a supremum. Then we apply the above result to
prove that if $X$ is a topological space with a partial order $\preceq $ for
which the order intervals are compact, $\mathcal{F}$ a nonempty commutative
family of monotone maps from $X$ into $X$ and there exists $c\in X$ such
that $c\preceq Tc$ for every $T\in \mathcal{F}$, then the set of common
fixed points of $\mathcal{F}$ is nonempty and has a maximal element. The
result, specialized to the case of Banach spaces gives a general fixed point
theorem that drops almost all assumptions from the recent results in this
area. An application to the theory of integral equations of Urysohn's type
is also given.
\end{abstract}

\maketitle





\section{Introduction}

M. R. Alfuraidan and M. A. Khamsi asked in \cite{AlKh} whether the classical
fixed point theorems for nonexpansive mappings still hold for
monotone-nonexpansive mappings. An interplay between the order and metrical
structure of the space turned out to be very fruitful in the years that
followed, including the Bishop-Phelps technique \cite{Ph}, counterparts of
Banach's contraction principle (see, e.g., \cite{Tu, RaRe, NiRo, Ja}) and
numerous applications, from linear and nonlinear matrix equations,
differential \ and integral equations, to game theory. For a recent account
of the theory we refer the reader to \cite{CaHe}.

Let $(M,\preceq )$ be a set with a partial order. A mapping $T:M\rightarrow
M $ is said to be monotone (or increasing) if $T(x)\preceq T(y) $ whenever $%
x\preceq y$. If in addition $(M,d)$ is a metric space then $T$ is said to be
monotone nonexpansive if
\begin{equation*}
d(T(x),T(y))\leq d(x,y)
\end{equation*}%
for every comparable $x,y\in M$ (i.e., $x\preceq y$ or $y\preceq x$). Fixed
point theory for nonexpansive mappings is broad and it is natural to study
its counterpart for monotone nonexpansive maps. It was initiated quite
recently in \cite{BaKh1}. In \cite{BiKh}, an analogue of the Browder-G\"{o}%
hde fixed point theorem in uniformly convex spaces was obtained. The next
step might be to attack a counterpart of classical Kirk's theorem in Banach
spaces with weak normal structure, and then to look for some counterexamples
of Alspach's type (for a background of metric fixed point theory, \cite{GoKi}
is recommended). Then we could examine a vast number of more or less
interesting generalizations of nonexpansive mappings from the
order-theoretical point of view that have been introduced for over fifty
years now (metric fixed point theory seems to be a real phenomenon here).
Such works have recently started to appear.

The aim of this note is to show that fixed point theory for monotone
nonexpansive-type mappings is not that rich. For its existential part, it
appears to be an application of the classical Knaster-Tarski theorem (also
known as the Abian-Brown theorem), see e.g., \cite[Theorem 2.1.1]{GrDu}, and
its generalization to commutative family of monotone mappings. We show in
Section 2 a general fixed point theorem with an independent proof that drops
almost all assumptions from the recent results in this area, see Corollary %
\ref{main3} and the comments after it. A counterpart of Kirk's theorem
follows as a very special case. In Section 3 we apply our results to the
theory of functional integral equations of Urysohn's type.

\section{Fixed point theorems}

The following lemma is crucial for our work. In this generality it is
probably new but some special cases follow for example from \cite[Prop. 2.32]%
{CaHe},\ combined with \cite[Prop. 1.1.7]{HeLa}. Let $X$ be a set with a
partial order $\preceq $. Consider the sets $[a,\rightarrow )=\{x\in
X:a\preceq x\}$ and $(\leftarrow ,b]=\{x\in X:x\preceq b\}$. Along this work
the concept of order intervals in $X$ will be restricted to the sets $%
[a,\rightarrow )$ and $[a,b]=[a,\rightarrow )\cap (\leftarrow ,b].$

Remember that a subset $J$ of a partially ordered set $X$ is directed if
each finite subset of $J$ has an upper bound in $J$.

\begin{lemma}
\label{compact}Let $X$ be a partially ordered set for which each family of
order intervals of the form $[a,b],[a,\rightarrow )$ with the finite
intersection property has a nonempty intersection. Then every directed
subset of $X$ has a supremum.
\end{lemma}

\begin{proof}
Let $J$ be a directed subset of $X$ and define
\begin{equation*}
K=\bigcap\limits_{x\in J}[x,\rightarrow ).
\end{equation*}%
For each $n$ and a sequence $\{x_{1},\ldots ,x_{n}\}\subseteq J$ an upper
bound in $J$ of the sequence is in $\bigcap_{i=1}^{n}[x_{i},\rightarrow )$
and so, from the finite intersection property, $K$ is nonempty. Take now $%
\{z_{1},\ldots ,z_{n}\}\subseteq K$ and consider $%
\bigcap_{i=1}^{n}[x_{i},z_{i}]$ which is again nonempty because the same
upper bound as above is still in this intersection. Therefore the family $%
\{[x,z]:x\in J,z\in K\}$ has the finite intersection property too and $%
K_{0}=\bigcap_{x\in J,z\in K}[x,z]$ is nonempty. Moreover, it is clear that $%
K_{0}\subseteq K$ and each element of $K_{0}$ is a lower bound of $K$. Hence
$K_{0}$ is a singleton and it is, in fact, the supremum of $J$.
\end{proof}

\begin{remark}
Notice that Lemma \ref{compact} holds for topological partially ordered spaces for
which order intervals are compact.
\end{remark}

If we combine Lemma \ref{compact} and the Knaster-Tarski theorem we obtain
immediately,

\begin{theorem}
\label{main1}Let $X$ be a topological space with a partial order $\preceq $
for which order intervals are compact and let $T:X\rightarrow X$ be
monotone. If there exists $c\in X$ such that $c\preceq T(c)$, then the set
of all fixed points of $T$ is nonempty and has a maximal element.
\end{theorem}

But the application of Lemma \ref{compact} is wider. Having it we obtain a
short and independent proof of the following strengthening of Theorem \ref%
{main1}.

\begin{theorem}
\label{main2}Let $X$ be a topological space with a partial order $\preceq $
for which order intervals are compact and let $\mathcal{F}$ be a nonempty
commutative family of monotone maps from $X$ into $X$. If there exists $c\in
X$ such that $c\preceq Tc$ for every $T\in \mathcal{F}$, then the set of
common fixed points of $\mathcal{F}$ is nonempty and has a maximal element.
\end{theorem}

\begin{proof}
Let%
\begin{equation*}
J_{0}=\{x\in X:x=T_{1}T_{2}...T_{n}c:T_{1},...,T_{n}\in \mathcal{F},n\in
\mathbb{N}\}\cup \{c\}.
\end{equation*}%
It is not difficult to see that $J_{0}$ is a directed set, $x\preceq Tx$ and
$Tx\in J_{0}$ for each $x\in J_{0}$ and $T\in \mathcal{F}$. Furthermore, if $%
\mathcal{J}$ is a chain of directed subsets of $X$ containing $J_{0}$ with
the above properties, then $\bigcup \mathcal{J}$ is a directed set with
these properties too, i.e., $x\preceq Tx$ and $Tx\in \bigcup \mathcal{J}$
for each $x\in \bigcup \mathcal{J}$ and $T\in \mathcal{F}$. By
Kuratowski-Zorn's lemma there exists a maximal directed set $J\subseteq X$
which contains $J_{0}$ with these properties and let $s=\sup J$ whose
existence follows from Lemma \ref{compact}. Then $x\preceq s$ and hence $%
x\preceq Tx\preceq Ts$ for every $x\in J$ and $T\in \mathcal{F}$. Thus each $%
Ts$ is an upper bound for $J$ and consequently $s\preceq Ts$ for $T\in
\mathcal{F}$. It follows from the maximality of $J$ that $s=Ts\in J$ for
each $T\in \mathcal{F}$ and clearly $s$ is a maximal fixed point of $%
\mathcal{F}$.
\end{proof}

\begin{remark}
Theorem \ref{main2} can be also deduced by combining Lemma \ref{compact}
with DeMarr's theorem \cite[Theorem 1]{De}.
\end{remark}

We are thus led to the following corollary.

\begin{corollary}
\label{main3}Let $X$ be a Banach space with a partial order $\preceq $
and let $\tau $ be a Hausdorff topology on $X$ such that the order intervals
are $\tau $-closed. Suppose $C$ is a (nonempty) $\tau $-compact subset of $X$
and $\mathcal{F}$ a (nonempty) commutative family of monotone maps from $C$
into $C$. Then $\mathcal{F}$ has a common fixed point if and only if there
exists $c\in C$ such that $c\preceq f(c)$ for every $f\in \mathcal{F}$.
Moreover, the set of common fixed points of $\mathcal{F}$ has a maximal
element.
\end{corollary}

Let us list the improvements of Corollary \ref{main3} over known results
till date:

\begin{itemize}
\item We do not impose any conditions on a Banach space $X$.

\item We consider an arbitrary Hausdorff topology $\tau $ while in most
papers the classical weak topology is considered or at least a topology for
which a Banach space $X$ satisfies a rather strong geometrical assumption
(the $\tau $-Opial condition).

\item A subset $C$ of $X$ is $\tau $-compact and in general need be neither
convex nor bounded.

\item We assume that $T:C\rightarrow C$ is monotone only and need be neither
monotone-nonexpansive nor continuous.

\item Rather than considering a single mapping $T$, we consider any family
of commuting monotone mappings.

\item And finally, apart from the existential result we obtain a qualitative
information about the set of fixed points that is sometimes helpful in
applications.
\end{itemize}

We conclude this section with a few special cases of Corollary \ref{main3}.

\begin{theorem}[{see {\protect\cite[Theorem 2.1]{BaKh1}}}]
Let $X$ be a Banach space. Let $\tau $ be a topology on $X$ such that $X$
satisfies the $\tau $-Opial condition. Let $\preceq $ be a partial order on $%
X$ such that order intervals are convex and $\tau $-closed. Let $C$ be a
bounded convex $\tau $-compact nonempty subset of $X$ and let $%
T:C\rightarrow C$ be a monotone nonexpansive mapping. Assume that there
exists $x_{0}\in C$ such that $x_{0}$ and $T(x_{0})$ are comparable. Then $T$
has a fixed point.
\end{theorem}

\begin{theorem}[{\protect\cite[Theorem 4.1]{BiKh}}]
Let $(X,\left\Vert \cdot \right\Vert ,\preceq )$ be a partially ordered
Banach space such that order intervals are closed and convex. Assume $X$ is
uniformly convex in every direction. Let $C$ be a nonempty weakly compact
convex subset of $X$ and let $T:C\rightarrow C$ be a monotone nonexpansive
mapping. Assume there exists $x_{0}\in C$ such that $x_{0}$ and $T(x_{0})$
are comparable. Then $T$ has a fixed point.
\end{theorem}

\begin{theorem}[{\protect\cite[Theorem 3.3]{AlKh}}]
Let $(X,\left\Vert \cdot \right\Vert ,\preceq )$ be a partially ordered
Banach space for which order intervals are convex and closed. Assume $X$ is
uniformly convex. Let $C$ be a nonempty convex closed bounded subset of $X$
and let $T:C\rightarrow C$ be a continuous monotone asymptotically
nonexpansive mapping. Then $T$ has a fixed point if and only if there exists
$x_{0}\in C$ such that $x_{0}$ and $T(x_{0})$ are comparable.
\end{theorem}

\begin{theorem}[{\protect\cite[Theorem 3.6]{AlKh}}]
Let $(X,\left\Vert \cdot \right\Vert ,\preceq )$ be a partially ordered
Banach space for which order intervals are convex and closed. Assume $X$ is
uniformly convex space that satisfies the monotone weak-Opial condition. Let
$C$ be a nonempty convex closed bounded subset of $X$ and let $%
T:C\rightarrow C$ be monotone asymptotically nonexpansive mapping. Then $T$
has a fixed point if and only if there exists $x_{0}\in C$ such that $x_{0}$
and $T(x_{0})$ are comparable.
\end{theorem}

\section{Examples of application}

Weak conditions assumed in the results obtained in the previous section
allow us to weaken conditions on examples of applications of monotonicity
results. One of these examples is provided, for instance, by \cite[Section 3]%
{BaKh1}.

Let $\Omega$ be a measure space, and consider the integral equation
\begin{equation}  \label{ie1}
x(t)=g(t)+\int_\Omega F(t,s,x(s))\; d\mu (s),\;\;\; t\in \Omega,
\end{equation}
where

\begin{itemize}
\item[i)] $g\in L^2(\Omega, {})$,

\item[ii)] $F\colon \Omega\times \Omega \times L^2(\Omega, {})\to {}$ is
measurable and monotone in its third coordinate.

\item[iii)] There exists a non-negative function $h(\cdot,\cdot)\in
L^2(\Omega\times \Omega)$ and $M\in [0,1/2)$ such that
\begin{equation*}
|F(t,s,x)|\le h(t,s)+M|x|,
\end{equation*}
where $t,s\in \Omega$ and $x\in L^2(\Omega, {})$.
\end{itemize}

As it is shown in \cite[Section 3]{BaKh1}, we can associate to this integral
equation the operator $J\colon L^2(\Omega, {})\to L^2(\Omega, {})$ given by
\begin{equation*}
(Jy)(t)=g(t)+\int_\Omega \tilde{F}(t)(y)(s)\; d\mu (s)
\end{equation*}
where $\tilde{F}(t)(y)(s) = F(t,s,y(s))$. Then it can be shown, in the same
terms as in \cite{BaKh1}, that $J$ sends the whole $L^2(\Omega, {})$ to a
closed ball of sufficiently large radius. Now, if we consider the weak
topology as $\tau$ in Theorem \ref{main2} it can be shown that order
intervals are closed with respect to this topology. Finally we need to
guarantee that there exists a $c\in L^2(\Omega, {})$ which is comparable
with $J(c)$. This will be furnished by the extra condition in the next
theorem.

\begin{theorem}
Under the above assumptions, we have that:

\begin{itemize}
\item[i)] The integral equation (\ref{ie1}) has a non-negative solution
provided we assume that $g(t)+\int_\Omega F(t,s,0)\;d\mu (s)\geq 0$ for
almost every $t\in \Omega$.

\item[ii)] The integral equation (\ref{ie1}) has a non-positive solution
provided we assume that $g(t)+\int_\Omega F(t,s,0)\;d\mu (s)\le 0$ for
almost every $t\in \Omega$.
\end{itemize}
\end{theorem}

\begin{proof}
Notice that the added condition implies that $J(0)\geq 0$ in i) and $J(0)\le
0$ in ii). The conclusion follows then after Theorem \ref{main1}.
\end{proof}

\begin{remark}
Comparing this example with the one provided in \cite[Section 3]{BaKh1} we
can notice that no nonexpansive condition is imposed on $F$ and this allows
us to weaken the conditions on the measure space which is no longer required
to be finite. In fact, our approach does not even require the measure space
to be $\sigma$-finite as it is the case for some other close examples to
this one in the literature as the one developed in \cite[Section 7.2.2]{CaHe}%
.
\end{remark}

\section{Acknowledgements}

R. Esp\'{\i}nola has been partially supported by DGES
(MTM2015-65242-C2-1-P). This work was developed while A. Wi\'{s}nicki was
visiting the University of Seville in the spring of 2017. He wishes to thank
the Department of Mathematical Analysis and IMUS (Instituto de
Investigaciones Matem\'{a}ticas de la Universidad de Sevilla) for
hospitality. The authors would like to thank the referee and, especially,
the Editor Professor Stanis\l aw Kwapie\'{n} for insightful comments on the
manuscript, providing us with an elegant proof of Theorem \ref{main2} and
suggestions for the preparation of the final version of this work.


\begin{thebibliography}{99}
\bibitem{AlKh} M. R. Alfuraidan and M. A. Khamsi, A fixed point theorem for
monotone asymptotically nonexpansive mappings, Proc. Amer. Math. Soc., to
appear.

\bibitem{BaKh1} M. Bachar and M. A. Khamsi, Fixed points of monotone
mappings and application to integral equations, Fixed Point Theory Appl.
(2015), 2015:110, 7 pp.

\bibitem{BiKh} B. A. Bin Dehaish and M. A. Khamsi, Browder and G\"{o}hde
fixed point theorem for monotone nonexpansive mappings, Fixed Point Theory
Appl. 2016, 2016:20, 9 pp.

\bibitem{CaHe} S. Carl and S. Heikkil\"{a}, Fixed point theory in ordered
sets and applications. From differential and integral equations to game
theory, Springer, New York, 2011.

\bibitem{De} R. DeMarr, Common fixed points for isotone mappings, Colloq.
Math. 13 (1964), 45--48.

\bibitem{GoKi} K. Goebel, W. A. Kirk, Topics in metric fixed point theory,
Cambridge University Press, Cambridge, 1990.

\bibitem{GrDu} A. Granas, J. Dugundji, Fixed point theory, Springer-Verlag,
New York, 2003.

\bibitem{HeLa} S. Heikkil\"{a} and V. Lakshmikantham, Monotone iterative
techniques for discontinuous nonlinear differential equations, Marcel
Dekker, Inc., New York, 1994.

\bibitem{Ja} J. Jachymski, The contraction principle for mappings on a
metric space with a graph, Proc. Amer. Math. Soc. 136 (2008), 1359--1373.

\bibitem{NiRo} J. J. Nieto and R. Rodr\'{\i}guez-L\'{o}pez, Contractive
mapping theorems in partially ordered sets and applications to ordinary
differential equations, Order 22 (2005), 223--239 (2006)

\bibitem{Ph} R. R. Phelps, Support cones in Banach spaces and their
applications, Advances in Math. 13 (1974), 1--19.

\bibitem{RaRe} A. C. M. Ran, M. C. B. Reurings, A fixed point theorem in
partially ordered sets and some applications to matrix equations, Proc.
Amer. Math. Soc. 132 (2004), 1435--1443.

\bibitem{Tu} M. Turinici, Fixed points for monotone iteratively local
contractions, Demonstratio Math. 19 (1986), 171--180.
\end{thebibliography}
\end{document}